\documentclass[11pt,reqno]{amsart}
\usepackage[utf8]{inputenc}
\usepackage{amsmath,amsthm,amssymb,amsopn}
\usepackage[pagebackref,colorlinks,linkcolor=red,citecolor=blue,urlcolor=blue,hypertexnames=false]{hyperref}
\usepackage{amsrefs}
\usepackage{amscd}

\input{xy} 
\xyoption{all}
\newdir{ >}{{}*!/-5pt/@{>}}
\usepackage{txfonts}

\newcommand\restr[2]{{% we make the whole thing an ordinary symbol
  \left.\kern-\nulldelimiterspace % automatically resize the bar with \right
  #1 % the function
  \vphantom{\big|} % pretend it's a little taller at normal size
  \right|_{#2} % this is the delimiter
  }} % Usar \restr{f}{A} para f|_A

\newcommand{\comment}[1]{}

\def\N{\mathbb{N}} 
\def\Z{\mathbb{Z}} 
 
\def\R{\mathbb{R}} 
\def\C{\mathbb{C}}

\def\topdf{\texorpdfstring}
\theoremstyle{plain}
\newtheorem{teo}[equation]{Theorem} 
\newtheorem{thm}[equation]{Theorem}
\newtheorem{lema}[equation]{Lemma}
\newtheorem{lem}[equation]{Lemma}
\newtheorem{coro}[equation]{Corollary} 
\newtheorem{prop}[equation]{Proposition}

\theoremstyle{definition}

\newtheorem{defi}[equation]{Definition}

\theoremstyle{remark}

  \numberwithin{equation}{section}

\newcommand{\cB}{\mathcal B}

\newcommand{\cO}{\mathcal O}
\newcommand{\cP}{\mathcal P}

\newcommand{\cR}{\mathcal R}
\newcommand{\cS}{\mathcal S}

\newcommand{\cV}{\mathcal V}

\def\fA{\mathfrak{A}}
\def\fB{\mathfrak{B}}

\newcommand{\lra}{\longrightarrow}
\newcommand{\iso}{\overset{\sim}{\lra}}

\newcommand{\onto}{\twoheadrightarrow}
\newcommand{\ol}{\overline}
\def\reg{\operatorname{reg}}

\def\sink{\operatorname{sink}}
\def\inf{\operatorname{inf}}

\def\triqui{\vartriangleleft}

\renewcommand{\path}{\cP}

\DeclareMathOperator*{\colim}{colim}

\title{Simplicity of $L^p$-graph algebras}
\author{Guillermo Corti\~nas}

\address{Guillermo Corti\~nas\\ Dep. Matem\'atica-IMAS, FCEyN-UBA\\ Ciudad Universitaria Pab 1\\
1428 Buenos Aires\\ Argentina}
\email{gcorti@dm.uba.ar}
\urladdr{http://mate.dm.uba.ar/\~{}gcorti}
\author{Diego Montero}
\author{Mar\'\i a Eugenia Rodr\'\i guez}
\address{Mar\'\i a Eugenia Rodr\'\i guez, Departamento de Ciencias Exactas\ \ \\ Ciclo B\'asico Com\'un,
Universidad de Buenos Aires, Ciudad Universitaria, (1428) Buenos Aires, Argentina}
\email{merodrig@dm.uba.ar}

\thanks{Corti\~nas is a CONICET researcher; he was partially supported by grants PIP 11220200100423 and PGC2018-096446-B-C21. Both Corti\~nas and Rodr\'\i guez were supported by grants UBACyT 20020170100256BA and PICT-2021-I-A-0071. Montero's research was carried out while supported by a CONICET PhD fellowship.}

\begin{document}
\begin{abstract}
For each $1\le p<\infty$ and each countable directed graph $E$ we consider the Leavitt path $\C$-algebra $L(E)$ and the $L^p$-operator graph algebra $\cO^p(E)$. We show that the (purely infinite) simplicity of $\cO^p(E)$ as a Banach algebra is equivalent to the (purely infinite) simplicity of $L(E)$ as a ring.
\end{abstract}

\maketitle

\section{Introduction}\label{sec:intro}

Let $E$ be a countable directed graph, $L(E)=L_\C(E)$ its complex Leavitt path algebra, and $p\in [1,\infty)$. The $L^p$-operator algebra of $E$, $\cO^p(E)$, introduced in \cite{we}, is universal for spatial representations of $L(E)$ in $L^p$-spaces; when $p=2$ it agrees with the graph $C^*$-algebra $C^*(E)$. A Banach algebra $\fA$ is \emph{simple} if it has exactly two two-sided closed ideals and \emph{simple purely infinite} (SPI) if $0\ne \fA\ne \C$ and for all $a,b\in \fA$ with $a\ne 0$ there are sequences $(x_n),(y_n)$ of elements of $\fA$ such that $x_nay_n\to b$. A ring $A$ is simple if it has exactly two two-sided ideals and is SPI if it is not zero or a division ring and for every $a,b\in A$ with $a\ne 0$ there exist $x,y\in A$ such that $xay=b$. The main result of the current paper is the following.

\begin{thm}\label{thm:intro}
Let $E$ be a countable graph and $p\in [1,\infty)\setminus\{2\}$.
\item[i)] $\cO^p(E)$ is a simple Banach algebra $\iff$ $L(E)$ is a simple ring.
\item[ii)] $\cO^p(E)$ is simple purely infinite as a Banach algebra if and only if $L(E)$ is simple purely infinite as a ring. 
\item[iii)] $\cO^p(E)$ is a simple but not simple purely infinite Banach algebra $\iff$ $L(E)$ is a simple but not purely infinite ring. 
\end{thm}
It is well-known \cite{lpabook}*{Chapter 3} that the situation of part iii) of the theorem above $E$ is acyclic. We also show (see Proposition \ref{prop:acyccount}) that $\cO^p(E)$ is almost finite (in the sense of \cite{chrismg}) for any countable acyclic graph $E$. Hence we deduce 
\begin{coro}\label{coro:introdicho}
If $\cO^p(E)$ is simple then it is either purely infinite or almost finite. 
\end{coro}
Both Theorem \ref{thm:intro} and Corollary \ref{coro:introdicho} were known for $p=2$ \cite{dt}*{Corollaries 2.13, 2.14 and 2.15}. Let $\cR_n$ be the graph consisting of a single vertex and $n$ loops. N. C. Phillips proved in \cite{chrisimple}*{Theorem 5.14} for $n\ge 2$ the $L^p$-Cuntz algebra $\cO^p_n=\cO^p(\cR_n)$, is simple purely infinite. Phillips' result, which is now a particular case of Theorem  \ref{thm:intro}, was the starting point for this article. Our proof that $E$ SPI implies $\cO^p(E)$ SPI is inspired by his arguments.  
An $L^p$-operator algebra $B$ is \emph{simple} if every nonzero contractive homomorphism to another $L^p$-operator algebra is injective. It was shown in \cite{we}*{Theorem 1.1} that $\cO^p(E)$ is a simple $L^p$-operator algebra if and only if $L(E)$ is a simple ring. Hence we deduce the following.

\begin{coro}\label{coro:introsimp}
$\cO^p(E)$ is simple as a Banach algebra $\iff$ it is simple as an $L^p$-operator algebra. 
\end{coro}

The rest of this paper is organized as follows. Some basic definitions and results about graphs, Leavitt path algebras and their $L^p$-completions are recalled in Section \ref{sec:prelis}. The notion of simple pure infiniteness for Banach algebras is introduced in Section \ref{sec:spiban}; for unital Banach algebras it is equivalent to algebraic SPI (Lemma \ref{lem:sat=p}). We also show (Corollary \ref{coro:corner}) that if a Banach algebra $\fA$ is SPI and $0\ne p=p^2\in\fA$, then $p\fA p$ is again SPI. The next three sections are mainly devoted to proving parts of Theorem \ref{thm:intro}. Section \ref{sec:acyc} deals with acyclic graphs. Proposition \ref{prop:acyccount} shows that for acyclic countable $E$ and $p\in [1,\infty)\setminus\{2\}$, $\cO^p(E)$ is a spatial AF-algebra in the sense of Phillips-Viola \cite{chrismg}. We use this and their results on the structure of ideals in such algebras \cite{chrismg2} to prove Proposition \ref{prop:simpacyc}, which says that if in addition $E$ is simple, then so is $\cO^p(E)$. In Section \ref{sec:spiOspiE} we establish Theorem \ref{teo:E-sip}, which says that $E$ is SPI whenever $\cO^p(E)$ is. The converse statement is Theorem \ref{teo:OE-sip}, proved in  Section \ref{sec:EspiOspi}. Its proof adapts and generalizes Phillips' arguments for the proof of \cite{chrisimple}*{Theorem 5.14}. Finally in Section \ref{sec:final} we put together the results of the previous section to prove Theorem \ref{thm:intro}.

\section{Preliminaries}\label{sec:prelis}
In this section we briefly recall (from \cite{lpabook} and \cite{we}) some of the basics of directed graphs, Leavitt path algebras and their $L^p$-completions.

An \emph{oriented graph}, herefrom simply a \emph{graph}, is a quadruple $E=(E^0,E^1,r,s)$ consisting of sets $E^0$ and $E^1$ of \emph{vertices} and \emph{edges}, and \emph{range} and \emph{source} functions $r,s\colon E^1\rightarrow E^0$. 
 
A vertex $v\in E^0$ is an \emph{infinite emitter} if $s^{-1}(v)$ is infinite, and is a \emph{sink} if $s^{-1}(v)=\emptyset$; otherwise we call $v$ a \emph{regular} vertex. We write $\sink(E),\inf(E),\reg(E)\subset E^0$ for the sets of sinks, infinite emitters, and regular vertices. We say that  $E$ is \emph{row-finite} if $\inf(E)=\emptyset$ and \emph{regular} $E^0=\reg(E)$. 

 A vertex $v$ is a \emph{source} if $r^{-1}(v)=\emptyset$.

A \emph{path} $\alpha$ is a (finite or infinite) sequence of edges $\alpha=e_1\ldots e_i\ldots$ such that $r(e_i)=s(e_{i+1})$  $(i\ge 1)$. For such $\alpha$, we write $s(\alpha)=s(e_1)$; if $\alpha$ is finite of \emph{length} $l$, we put $|\alpha|=l$ and $r(\alpha)=r(e_l)$. If $\alpha$ and $\beta$ are paths with $|\alpha|<\infty$, we write
\begin{equation}\label{orderpath}
    		\alpha\ge\beta \iff \exists\ \gamma \text{ such that } r(\alpha)=s(\gamma) \text{ and } \beta=\alpha\gamma.
\end{equation}
Here $\alpha\gamma$ is the path obtained by concatenation.

Vertices are considered as paths of length $0$. 
We write $\mathcal{P}=\cP(E)$ for the set of finite paths, $\mathcal{P}_n$ for the set of paths of length $n$, and, if $v\in E^0$, $\cP_{n,v}$ for the paths $\alpha$ of length $n$ with $r(\alpha)=v$ . 

Adding a basepoint $0$ to the set $\cP$ one obtains a pointed semigroup $\cP_+$
where the product of two paths is their concatenation whenever it is defined, and is $0$ otherwise. By adding a formal inverse $\alpha^*$ for any $\alpha\in\cP$, 
one obtains a pointed inverse semigroup $\cS(E)$ whose elements are $0$ and pairs of the form $\alpha\beta^*$ with $r(\alpha)=r(\beta)$. 
To any graph $E$ and field $\ell$, one associates an $\ell$-algebra $L_\ell(E)$, the \emph{Leavitt path algebra} of $E$ (\cite{lpabook}*{Definition 1.2.3}), which is universal for tight representations of $\cS(E)$ in $\ell$-vector spaces \cite{we}*{Lemma 3.1}. In this paper we only consider the case $\ell=\C$; we write $L(E)=L_\C(E)$. Path length induces a $\Z$-grading $L(E)=\displaystyle{\bigoplus_{n\in\Z}}L(E)_n$ such that $|\alpha\beta^*|=|\alpha|-|\beta|$. The homogeneous component of degree zero carries a filtration 
\begin{equation}\label{eq:filtle0}
L(E)_0=\bigcup_{n=0}^\infty L(E)_{0,n},
\end{equation}
where $L(E)_{0,n}$ is spanned by the $\alpha\beta^*$ with $\alpha,\beta\in\cP_m$, $m\le n$. If $E$ is row-finite, the algebra $L(E)_{0,n}$ is isomorphic to a direct sum of matrix algebras, indexed by paths in $E$, as follows
\begin{equation}\label{eq:le0r}
\left(\bigoplus_{r=0}^{n-1}\bigoplus_{v\in\sink(E)}M_{\cP_{r,v}}\right)\bigoplus\bigoplus_{v\in\reg(E)}M_{\cP_{n,v}}\iso L(E)_{0,n}.
\end{equation}
The isomorphism above sends the matrix unit $\epsilon_{\alpha,\beta}$ to $\alpha\beta^*$. 

We say that a graph $E$ is \emph{countable} if both $E^0$ and $E^1$ are countable. For countable $E$ and $p\in[1,\infty)$ the \emph{$L^p$-operator algebra} $\cO^p(E)$ is defined in \cite{we}*{Definition 7.4}. It comes equipped with an algebra homomorphism $\rho:L(E)\to \cO^p(E)$ that is injective \cite{we}*{Proposition 4.11} and universal initial among all $L^p$-representations that are spatial in the sense of Phillips \cite{lpcuntz}. The case $p=2$ recovers the usual Cuntz-Krieger graph $C^*$-algebra; we have $\cO^2(E)=C^*(E)$ \cite{we}*{Proposition 7.9}.
The algebra $\cO^p(E)$ carries a continuous gauge action of the circle group $\mathbb{S}^1$ \cite{we}*{Lemma 7.8} by isometric isomorphisms $\gamma_z$ ($z\in \mathbb{S}^1$). We write
\begin{equation}\label{eq:open}
 \cO^p(E)_{n}=\{x\in\cO^p(E): \gamma_z(x)=z^nx\} 
\end{equation}
for the homogeneous component of degree $n$ of the associated $\Z$-grading. It is the closure of the image of $L(E)_n$ in $\cO^p(E)$ \cite{we}*{Proposici\'on 3.1.7 (d)}. In particular, writing $\cO^p(E)_{0,n}$ for the image of $L(E)_{0,n}$, we have
\[
\cO^p(E)_{0}=\overline{\bigcup_{n\ge 0} \cO^p(E)_{0,n}}.
\]
The map $L(E)\to L(E)_n$ that sends an element to its homogeneous component of degree $n$ extends to an idempotent operator $\Phi_n:\cO^p(E)\to \cO^p(E)$ with image $\cO^p(E)_n$ defined by
\begin{equation}\label{map:phin}
\Phi_n(a)=\frac{1}{2\pi}\int_0^{2\pi}e^{-in\theta}\gamma_{e^{i\theta}}(a)d\theta.     
\end{equation}

Some basic properties of $\Phi_n$, analogous to those well-known for $p=2$, were established in \cite{etesis}*{Secci\'on 3.1}. In particular we have the following.

\begin{lem}\label{lem:record}
Let $E$ be a countable graph, $p\in [1,\infty)$, $a\in\cO^p(E)$, $m,n\in\Z$ and $b\in\cO^p(E)_m$. 
\item[i)] \cite{etesis}*{Lema 3.1.6 (a)} $\Phi_n(ab)=\Phi_{n-m}(a)b$, $\Phi_n(ba)=b\Phi_{n-m}(a)$.
\item[ii)] \cite{etesis}*{Corolario 3.1.9} If $a\ne 0$ then there exists $r\in\Z$ such that $\Phi_r(a)\ne 0$.
\end{lem}

A finite path $\alpha$ is \emph{closed} if $s(\alpha)=r(\alpha)$. A closed path $\alpha=e_1\dots e_n$ is a \emph{cycle} if in addition $s(e_i)\ne s(e_j)$ if $i\ne j$. An \emph{exit} of a cycle $\alpha$ as above is an edge $f$ such that there exists $1\le i\le n$ such that $s(f)=s(e_i)$ and $f\ne e_i$. A graph $E$ is \emph{cofinal} if for every $v\in E^0$ and each cycle $c$ there is a path starting at $v$ and ending at some vertex in $c$. 

\begin{defi}\label{defi:Esimple}
A graph $E$ is \emph{simple} if it is cofinal and every cycle in $E$ has an exit. A simple graph $E$ is \emph{simple purely infinite } if in addition $E$ has at least one cycle. 
\end{defi}
In the following theorem we record some known equivalences between simplicity conditions on $E$, $L(E)$ and $\cO^p(E)$.

\begin{thm}\label{thm:record}
Let $E$ be a graph and $p\in [1,\infty)$. 
\item[i)] \cite{lpabook}*{Theorem 2.9.7} $E$ is simple if and only if $L(E)$ is simple. If $E$ is countable this is further equivalent to: every nonzero spatial representation of $\cO^p(E)$ is injective \cite{we}*{Theorem 1.1}. 
\item[ii)]\cite{lpabook}*{Theorem 3.1.10} $E$ is SPI if and only if $L(E)$ is. 
\end{thm}

\section{Simple purely infinite  Banach algebras}\label{sec:spiban}
\begin{defi}\label{defi:spi}
Let $\fA$ be  Banach algebra.  $\mathfrak{A}$ is  \emph{simple} if it has exactly two closed two-sided ideals. $\fA$ is \emph{simple purely infinite} if $ 0\ne \mathfrak{A} \neq \C$ and for every  $a,b \in \mathfrak{A}$ with $a\ne 0$ there are sequences $\{ x_n\}_{n\in \N}$ and $ \{ y_n\}_{n\in \N} $ in $\mathfrak{A}$ such that $$ \lim_{n \to \infty} x_n a y_n  = b .$$
\end{defi}
\begin{lema}\label{lem:sat=p}
    Let $\fA$ be an SPI Banach algebra and $a,p\in \fA\setminus\{0\}$ with $p$ idempotent.
Then there are $s\in p\fA$ and $t\in\fA p$ such that $sat=p$.
\end{lema}
\begin{proof}
Because $\fA$ is SPI, there exist $s_n,t_n\in \fA$ $(n\ge 1)$ such that $x_n=s_nat_n\to p$. Upon multiplying by $p$ on both sides, we may assume $s_n=ps_n$ and $t_n=t_np$ so that $x_n\in \fB:=p\fA p$. Now $\fB$ is a unital Banach algebra, and thus its group of units is open. Hence for $n>>0$ there is $b\in\fB$ with $bx_n=p$, and we may take $s=bs_n$ and $t=t_n$. 
\end{proof}
For $n\ge 1$, consider the \emph{Cohn algebra} \cite{cohn}
\[
C_n=\C\{x_1,\dots,x_n,y_1,\dots,y_n\}/\langle y_ix_j-\delta_{i,j}: 1\le i,j\le n\rangle.
\]
\begin{lem}\label{lem:spipi} Let $\fA$ and $p$ be as in Lemma \ref{lem:sat=p}. Then there
exists a unital $\C$-algebra homomorphism $C_2\to p\fA p$.
\end{lem}
\begin{proof}
 We must show that there exist elements $s_i,t_i\in p\fA p$, $i=1,2$, such that $t_is_j=\delta_{i,j}p$, or equivalently, that the right module $p\fA\oplus p\fA$ embeds into $p\fA$ as a direct summand. Assume first that $\fA$ is unital with unit $p$. Then by Lemma \ref{lem:sat=p}, $\fA$ is SPI as a ring, and so the lemma follows from 
 \cite{agp}*{Proposition 1.5}. The case when $p$ is not a unit follows from Lemma \ref{lem:sat=p} and the argument of \cite{lpabook}*{Proposition 3.8.8, proof that (2)$\Rightarrow$(3)}.
\end{proof}
\begin{coro}\label{coro:corner}
If $\fA$ is an SPI Banach algebra and $p\in\fA$ a nonzero idempotent, then the Banach subalgebra $p\fA p$ is again SPI. 
\end{coro}

\begin{proof}
It is clear that $\fB:=p\fA p\ne 0$ and that if $0\ne a,b\in\fB$ then there are sequences $x_n,y_n\in\fB$ as in Definition \ref{defi:spi}. Moreover, it follows from Lemma \ref{lem:spipi} that $\fB$ is not a division ring; this concludes the proof. 
\end{proof}

\section{The \topdf{$L^p$}{Lp}-operator algebra of a simple acyclic graph}\label{sec:acyc}

\begin{lema}\label{lem:acycfin}
Let $F$ be a finite acyclic graph and $p\in [1,\infty)\setminus\{2\}$. Let $||\ \ ||:L(F)=\cO^p(F)\to\R_{\ge 0}$ be the norm. The following are equivalent for a faithful representation $\rho:L(F)\to \cB(L^p(X))$.
\item[i)] $\rho$ is spatial.
\item[ii)] $\rho$ is isometric for the norm $||\ \ ||$.
\end{lema}
\begin{proof}
 Write $\cP_v$ for the set of all paths ending at a vertex $v\in F^0$. Because $F$ is acyclic, for every $v\in\reg(F)$ we may write $v=\sum\alpha\alpha^*$ where the sum runs over all paths starting at $v$ and ending at a sink. It follows from this that $\rho$ is spatial if and only if for all $v\in\sink(F)$ and $\alpha,\beta\in\cP_v$, $\rho(\alpha\beta^*)$ is a spatial partial isometry. By \cite{lpabook}*{Theorem 2.6.17}  the $\C$-linear map 
 \[
 \phi:L(F)\to \bigoplus_{v\in\sink(F)}M_{\cP_v},\ \ \phi(\alpha\beta^*)=\epsilon_{\alpha,\beta}\ \ (\alpha,\beta\in\cP_v, v\in\sink(F))
 \]
 is an isomorphism of algebras. Hence $\rho$ is spatial if and only if the representation $\rho'=\rho\circ\phi^{-1}:\displaystyle{\bigoplus_{v\in\sink(F)}}M_{\cP_v}\to \cB(L^p(X))$ it corresponds to under the isomorphism above maps the matrix units $\epsilon_{\alpha,\beta}$ to spatial partial isometries, which precisely means that the latter restrictions are spatial in the sense of \cite{lpcuntz}*{Definition 7.1}. 
Thus by \cite{lpcuntz}*{Theorem 7.2}, $\rho$ is spatial if and only if the restriction of $\rho'$ to each of the summands $M_{\cP_v}=\cB(\ell^p(\cP_v))$ is isometric with respect to the operator norm. In particular, if the latter condition holds, then $\rho'$ maps the identity matrices of the summands to orthogonal spatial idempotents. It follows that if $$L(F)\owns x=\displaystyle{\sum_{v\in \sink(F)}}x_v$$ with $x_v\in\phi^{-1}(M_{\cP_v})$, then
$||\rho(x)||=\max\{||x_v||_{\cB(\ell^p(\cP_v))}:v\in\sink(F)\}$.
In particular this applies to the representation defining $\cO^p(F)$; thus  $||x||_{\cO(F)}=||\rho(x)||$.
\end{proof}
We introduce some notation that will be used below. Let $F\subset E$ be a finite complete subgraph and $v\in F^0$. For $v\in \reg(F)\cap\inf(E)$, we consider the following elements of $L(E)$
\[
m_v^F=\sum_{e\in F,\ s(e)=v}ee^*, \ \ q_v^F=v-m_v^F.
\]
\begin{prop}\label{prop:acyccount}
Let $E$ be a countable acyclic graph and $p\in [1,\infty)\setminus\{2\}$. Then $\cO^p(E)$ is a spatial $AF$-algebra in the sense of\ \ \cite{chrismg}*{Definition 9.1}.   
\end{prop}
\begin{proof}
   For each finite complete subgraph $F\subset E$, let $\tilde{F}=F(\reg(E)\cap\reg(F))$ be as in 
    \cite{lpabook}*{Definition 1.5.16}. Set $Y^F=\reg(F)\cap\inf(E)$. The graph $\tilde{F}$ has vertices $\tilde{F}^0=F^0\sqcup \{v': v\in Y^F\}$ and edges $\tilde{F}^1=F^1\sqcup \{e':r(e)\in Y^F\}$; its source and restriction maps extend those of $F$, and we have $s(e')=s(e)$ and $r(e')=r(e)'$.  It follows from the proof of \cite{lpabook}*{Theorem 1.5.18} that the inclusion $F\subset E$ induces a $*$-algebra homomorphism 
    $\iota_{F}:L(\tilde{F})\to L(E)$ determined by 
    \begin{gather*}
    \iota_{F}(v)=\left\{\begin{matrix} v& v\notin Y^F\\ m_v^F & v\in Y^F\end{matrix}\right., \ \ \iota_{F}(v')=q_v^F,\\
    \iota_{F}(e)=\left\{\begin{matrix} e& r(e)\notin Y^F\\ em_{r(e)}^F & r(e)\in Y^F\end{matrix}\right., \ \ \iota_{F}(e')=eq_{r(e)}^F.
    \end{gather*}    
Moreover 
$\iota_F$ is injective by \cite{lpabook}*{Theorems 1.5.8 and 1.6.10}, and we have $L(E)=\colim_FL(\tilde{F})$. Because $E$ is countable, its finite complete subgraphs are countably many, and thus we may choose a cofinal ascending chain $\{F_n:n\in\N\}$ of finite complete subgraphs of $E$ \cite{stex}, so that $L(E)=\colim_{n\in \N}L(\tilde{F}_n)$. Let $\rho:L(E)\to\cB(L^p(X))$ be a faithful spatial representation such that $\ol{\rho(L(E))}=\cO^p(E)$. Let $F\subset E$ be a finite complete subgraph. It follows from the explicit description of $\iota_F$ above that $\rho_F:=\rho\circ\iota_F$ is an (injective) spatial representation. Since by \cite{lpabook}*{Proposition 1.5.21}, $\tilde{F}$ is acyclic, it follows from Lemma \ref{lem:acycfin} that $\rho_F$ defines an isometric inclusion $O^p(\tilde{F})\to \cO^p(E)$. Therefore if $F\subset G\subset E$ is another finite complete subgraph, then $\cO^p(\tilde{F})\to\cO^p(\tilde{G})$ is isometric, which again by Lemma \ref{lem:acycfin} implies that it is spatial in the sense of \cite{chrismg}*{Definition 8.13}. Summing up,
$\{\cO^p(\tilde{F}_n):n\in\N\}$ is a spatial $L^p$-$AF$ direct system in the sense of \cite{chrismg}*{Definition 9.1} and $\cO^p(E)=\colim_n\cO^p(\tilde{F}_n)$ as Banach algebras, which shows that $\cO^p(E)$ is a spatial $L^p$-$AF$-algebra in the sense of \cite{chrismg}.
\end{proof}

\begin{prop}\label{prop:simpacyc}
 Let $E$ be a countable acyclic graph and $p\in [1,\infty)\setminus\{2\}$. If $E$ is simple, then $\cO^p(E)$ is a simple Banach algebra.
\end{prop}
\begin{proof}
Let $I\triqui\cO^p(E)$ be a closed two-sided ideal. By \cite{chrismg2}*{Theorem 3.5 and Lemma 3.8} and the proof of Proposition \ref{prop:acyccount}, there exists a finite complete subgraph $F\subset E$ such that $0\ne\cO^p(\tilde{F})\cap I=L(\tilde{F})\cap I$. Because $E$ is simple, so is $L(E)$, whence $L(E)\subset I\subset \cO^p(E)$ and therefore $I=\cO^p(E)$ since $I$ is closed and 
$\overline{L(E)}=\cO^p(E)$.
\end{proof}

\section{\topdf{$\cO^p(E)$}{Op(E)} SPI implies \topdf{$E$}{E} SPI}\label{sec:spiOspiE}

\begin{teo}\label{teo:E-sip}
Let $E$ be a countable graph and $p\in [1,\infty)$. If $\mathcal{O}^p(E)$ is an SPI Banach algebra, then $E$ is an SPI graph.
\end{teo}
\begin{proof}
Since $\mathcal{O}^p(E)$ is simple as a Banach algebra, then it is simple as an $L^p$-operator algebra, and therefore  $L(E)$ is simple by \cite{we}*{Theorem 10.1}. Hence $E$ is simple by the simplicity theorem \cite{lpabook}*{Theorem 2.9.1}; to show that $E$ is SPI, we must prove that every vertex connects to a cycle. 

By Corollary \ref{coro:corner} it suffices to show that if a vertex $w\in E^0$ does not connect to any cycle, then $w\mathcal{O}^p(E)w$ is not purely infinite simple. Let $v\in E^0$; following \cite{lpabook}*{top of page 63}, we write $v\le w$ if there is a path $\mu$ with $s(\mu)=w$ and $r(\mu)=v$. 
Let $H$ the graph with $H^0 := \{ v \in E^0 \ : \ w \geq v \}$, $H^1 := s_E^{-1}(H^0)$ and $s_H, r_H$ the restrictions to $H^1$ of $s_E$ and $r_E$. This is the graph considered in the proof of \cite{aap}*{Proposition 9}, where the definition of $\le $ is reversed. By \cite{aap}, $wL(E)w=wL(H)w$, and therefore 
\begin{equation}\label{eq:mismeq}
    w\cO^p(E)w=\ol{wL(E)w}=\ol{wL(H)w}=w\cO^p(H)w,
\end{equation}
since the map $A\mapsto wAw$ is contractive. 

Remark that the family 
\begin{equation}\label{eq:approx1}
\{v_F=\sum_{v\in F} w\ \ |\ \ F\subset H^0\text{ finite}\}    
\end{equation}
is a net of uniformly bounded approximate units in $\cO^p(H)$. Next observe that by definition of $H$, for every $v\in H^0$ there is a path $\alpha_v$ such that $s(\alpha_v)=w$, $r(\alpha_v)=v$. Hence if $w\in F\subset H^0$ is finite, then 
\begin{equation}\label{eq:wfull}
v_F=\sum_{v\in F}\alpha_v^*\alpha_v\in v_F\cO^p(H)w\cO^p(H)v_F=(v_F\cO^p(H)v_F)w(v_F\cO^p(H)v_F).
\end{equation}
Hence using \eqref{eq:wfull} at the second step and \eqref{eq:mismeq} at the last, we obtain the following identities for the monoids of Murray-von Neumann equivalence classes of idempotent matrices
\begin{equation}\label{eq:vtwosides}
\cV(\cO^p(H))=\colim_{F\owns w}\cV(v_F\cO^p(H)v_F)=\colim_{F\owns w}\cV(w\cO^p(H)w)=\cV(w\cO^p(E)w).
\end{equation}
It follows from Proposition \ref{prop:acyccount} that for the left hand side of \eqref{eq:vtwosides} the group completion map $\cV(\cO^p(H))\to \cV(\cO^p(H))^+=K_0(\cO^p(H))$ is injective. However the right hand side is the $\cV$-monoid of the unital simple purely infinite ring $R=w\cO^p(E)w$ and thus by \cite{agp}*{Proposition 2.1 and Corollary 2.2}, the subsemigroup $\cV(R)\supset G=\cV(R)\setminus\{0\}$ is a group and the group completion of $\cV(R)$ is the map $\cV(R)\to G$ that restricts to the identity on $G$ and maps $0$ to the zero element of $G$. In particular $\cV(R)\to \cV(R)^+$ is not injective. This concludes the proof.
\end{proof}

\section{\topdf{$E$}{E} SPI implies \topdf{$\cO^p(E)$}{Op(E)} SPI}\label{sec:EspiOspi}

The purpose of this section is to prove the following. 

\begin{teo}\label{teo:OE-sip}
Let $p\in [1,\infty)$ and let $E$ be a purely infinite simple countable graph. Then $\mathcal{O}^p(E)$ is a simple purely infinite Banach algebra.
\end{teo}

We divide the proof in two parts, with several lemmas in between.

\begin{proof}[Proof of Theorem \ref{teo:OE-sip}, part 1:] \emph{reduction to the case when $E$ is row-finite without 
\goodbreak
\noindent sources.} Assume the theorem known for row-finite graphs without sources. If $E$ is row-finite, let $E_{\mathfrak{r}}$ be the graph that results upon source removal, and $\phi_\mathfrak{r} : \mathcal{O}^p(E) \to \mathcal{O}^p(E_\mathfrak{r})$ the natural inclusion of \cite{we}*{Section 8}. For each finite subset $F\subset E^0$, let $v_F$ be as in \eqref{eq:approx1}.
One checks that that $v_F L(E_\mathfrak{r}) v_F \subseteq L(E)$; it follows that $v_F \mathcal{O}^p(E_\mathfrak{r}) v_F \subseteq \mathcal{O}^p(E)$. For each $0\ne a \in \mathcal{O}^p(E)$ we may choose an $F$ such that $v_Fav_F\neq 0$. Let $v\in E^0$. Since $E_\mathfrak{r}$ is purely infinite simple without sources, by assumption there are $x', y' \in \mathcal{O}^p(E_\mathfrak{r})$ such that $x' v_F a v_F y' = v$. So for $x" = v_Fx' v_F $ and $y" = v_F y' v_F$ we have $x",y"\in \mathcal{O}^p(E)$ and $x" a y" = v\in L(E)$. If now $b\in\cO^p(E)$ then for each $n$ there exist $b_n\in L(E)$ with $||b-b_n||<1/n$ and because $E$ is SPI, also $u_n,z_n\in L(E)$ with $u_nvz_n=b_n$. Summing up, the sequences $x_n=u_nx"$, $y_n=y"z_n$ satisfy $x_nay_n\to b$. We have proven that if the theorem holds for row-finite graphs without sources then it holds for all row-finite graphs.  If now $E$ has infinite emitters, we may consider its row-finite desingularization $E\to E_{\mathfrak{d}}$ as in \cite{we}*{Section 7}; again one checks that $v_FL(E_{\mathfrak{d}})v_F\subset L(E)$ for each finite $F\subset E^0$ and a similar argument as above shows that if the theorem holds for $E_{\delta}$ then it does also for $E$.
\end{proof}

\begin{lema}\label{lem:caminos}
Let $E$ be a purely infinite simple graph and $\mathcal{V} \subseteq E^0$ a finite subset such that each $v \in \mathcal{V}$ is the base of, at least, one cycle. Then, for every $m\ge 1$ and $v\in\cV$ there exist $\ell\ge 1$  and $m$ distinct closed paths $\gamma_1^{v}, \dots, \gamma_m^{v}$ of length $\ell$ based at $v$ .
\end{lema}
\begin{proof}
For each $v \in \mathcal{V}$, take a cycle $\alpha_v$ based at $v$. Since every cycle has an exit and $E$ is cofinal, there is a closed path $\beta_v$ based at $v$ such that $\alpha_v^* \beta_v = \beta_v^* \alpha_v = 0$. Thus distinct words on $\alpha_v$ and $\beta_v$ give distinct  closed paths based at $v$. Hence for $n_v$ sufficiently large there are closed paths $\delta_1^v, \dots, \delta_m^v$ based at $v$, all of the same length $n_v$. (For example we may take $\delta_i^v=\alpha_v^i\beta_v\alpha_v^{m-i}$ and $n_v=m|\alpha_v|+|\beta_v|$.) 

Next, let $\ell$ be the least common multiple of the $n_v$ with $v \in \mathcal{V}$, and set $\ell_v=\ell/n_v$. Then the closed paths
$\{\gamma_i^v=(\delta_i^v)^{\ell_v}: v\in\cV,\ \ 1\le i\le m\}$ satisfy the required conditions.
\end{proof}

\begin{lem}\label{lem:endo}
Let $E$ be a row-finite countable graph and $p\in [1,\infty)$. For $r \in \N$ let $\psi_r : \mathcal{O}^p(E) \to \mathcal{O}^p(E)$ be the linear map defined by $$ \psi_r(a) = \sum_{\gamma \in \path_r(E)} \gamma a \gamma^*.$$
Then: 
\begin{enumerate}
    \item $\psi_r$ is well-defined. 
    \item $\psi_r$ is contractive. 
    \item For every $a \in \mathcal{O}^p(E)$ and every $x \in \cO^p(E)_{0,r}$, $\psi_r(a) x = x \psi_r(a)$.
    \item $\psi_r(\cO^p(E)_n) \subseteq \cO^p(E)_n$.
    \item If $a,b \in \mathcal{O}^p(E)$ are such that $v a = a v $ or $v b = b v$ for every $v \in E^0$, then $\psi_r(ab) = \psi_r(a) \psi_r(b)$.
    
\end{enumerate}
\end{lem}
\begin{proof}
By \cite{we}*{pagraph before Proposition 7.5}, there is an injective, nondegenerate spatial representation $\rho:L(E)\to \cB(L^p(X))$ such that $\cO^p(E)=\overline{\rho(L(E))}$. By \cite{we}*{Remark 4.5}, there is a family of disjoint measurable subsets $(X_\gamma)_{\gamma\in\cP_r(E)}$ of $X$ such that $\rho(\gamma\gamma^*)$ is the canonical projection $L^p(X)\onto L^p(X_\gamma)\subset L^p(X)$. Hence for any finite subset $F\subset\cP_r$, we have $||\sum_{\gamma\in F}\gamma a\gamma^*||=\max_{\gamma\in F}||\gamma a\gamma^*||\le ||a||$. Thus $\psi_r$ is well-defined and contractive, proving (1) and (2). A standard density argument shows that the general case of part (3) follows from the case when $x\in L(E)_{0,r}$. Recall $L(E)_{0,r}\subset L(E)_0$ is the subspace generated by the elements of the form $\gamma\gamma^*$ where either $|\gamma|=r$ of $|\gamma|<r$ and $r(\gamma)\in\sink(E)$; hence it suffices to check that (3) holds for such generators, and this is straightforward. Part (4) is immediate from the fact that $|\gamma^*|=-|\gamma|$, and part (5) follows from the fact that if $\alpha,\beta\in\cP_r$ then 
$\alpha^*\beta=\delta_{\alpha,\beta}r(\alpha)$.
\end{proof}

\begin{lema} \label{lem:caso0}
Let $E$ be a purely infinite simple, countable, row-finite graph, and $p\in [1,\infty)$.  Let $0\ne a \in \cO^p(E)_0 $. Then there are $h \in \N_0$, $x\in \cO^p(E)_{-h}$, $y \in \cO^p(E)_{h}$ and $v\in E^0$ such that $x a y = v $.
\end{lema}

\begin{proof}
Let $\langle a \rangle$ be the closed two-sided ideal of $\cO^p(E)_0$ generated by $a$. Because $\cO^p(E)_{0,n}=L(E)_{0,n}$ is matricial for all $n$, and because $a \neq 0$ by hypothesis, by (\cite{chrismg2}*{Theorem 3.5}) , there exists $\ell\in \N $ such that $\langle a \rangle \cap \cO^p(E)_{0,\ell}\ne 0 $. By definition of $\cO^p(E)_{0,\ell}$ there exists $\alpha \in \path_\ell(E)$ such that $\alpha \alpha^* \in \langle a \rangle$. Since $E$ contains at least one cycle, using the Cuntz-Krieger relation CK2 \cite{lpabook}*{Definition 1.2.3} and the hypothesis that $E$ is row-finite,   upon increasing $\ell$ if necessary, we may assume that $r(\alpha)$  is the base of at least one cycle.  Let $d\ge 1$ and $b,c\in \cO^p(E)^d_0$  such that
$$ \left\| \alpha \alpha^* - \sum_{k = 1}^{d} b_k a c_k \right\| < \frac{1}{8}.$$ 
Without loss of generality we may assume that $\alpha \alpha^* b_k = b_k$ and $c_k \alpha \alpha^*  = c_k$. Set $M = \left(1+ ||b||_1 \right)\left(1+ ||c||_1\right)$. By density, there are $m\ge \ell$, $a_0\in \cO^p(E)_{0,m}$, $b^0\in \alpha\alpha^*\cO^p(E)^d_{0,m}$ and $c^0\in \cO^p(E)^d_{0,m}\alpha\alpha^*$ such that 
$$\| a_0 - a \| < \frac{1}{8M},\ \  ||b^0-b||_1<\frac{1}{8} (\frac{1}{(1+||a||)(1+||c||_1)}),\ \ ||c^0-c||_1<\frac{1}{8(1+||a||)\cdot (1+||b||_1)}.$$
In particular $||b-b^0||_1<1$ and $||c-c^0||_1<1$, and therefore 
\begin{equation}\label{ineq:M}
||b^0||_1 ||c^0||_1<M.  
\end{equation}
Moreover we have 
\begin{gather*}
||(\sum_{k=1}^d    b_k^0a_0c_k^0)-\alpha\alpha^*||\le ||\sum_{k=1}^d(b_k^0-b_k)a_0c_k^0||+||\sum_{k=1}^db_k(a_0-a)c_k^0||+
\\ ||\sum_{k=1}^db_ka(c_k^0-c_k)||+||(\sum_{k=1}^db_kac_k)-\alpha\alpha^*||<\frac{1}{2}
\end{gather*}
Hence  $ \sum_{k=1}^d b_k^0 a_0 c_k^0$ is an element of the open ball of radius $\frac{1}{2}$ centered at the unit of the unital Banach algebra $\alpha\alpha^*\cO^p(E)_{0,m}\alpha\alpha^*$. Thus there exists $z \in \alpha \alpha^*\cO^p(E)_{0,m}\alpha \alpha^*$, such that $ z\left( \sum_{k=1}^d b_k^0 a_0 c_k^0 \right)=\left( \sum_{k=1}^d b_k^0 a_0 c_k^0 \right)z = \alpha \alpha^*$. Moreover a standard argument shows that
\begin{equation}\label{ineq:z<2}
||z||<2. 
\end{equation}

Because $E$ is row-finite, the set $\mathcal{V} := \{ r(\beta) \ : \ s(\beta) = r(\alpha) \text{ and } |\beta| = m \}$ is finite, and because $r(\alpha)$ is the base of a cycle, the same is true of every $w \in \mathcal{V}$, by cofinality. Hence by Lemma \ref{lem:caminos} for every $w\in\cV$ there are $d$ closed paths based at $w$, $ \gamma_1^w, \dots, \gamma_d^w$, all of the same length $n$. Set $t =  \psi_m (\sum_{w \in \mathcal{V}} (\gamma_1^w)^*)$, $u =  \psi_m (\sum_{w \in \mathcal{V}} (\gamma_1^w))$, $f_{j,k} =  \psi_m (\sum_{w \in \mathcal{V}} \gamma_j^w(\gamma_k^w)^*)$, $x_0=z t( \sum_{k=1}^d b_k^0 f_{1,k})$ and $y_0 = ( \sum_{k=1}^d f_{k,1}c_k^0) u$. Observe that $x_0 \in \cO^p(E)_{-n}$ and $y_0 \in \cO^p(E)_{n}$. Using (3) and (5) of Lemma \ref{lem:endo}, we obtain
\begin{equation}\label{eq:daalpha}
    x_0 a_0 y_0 =  z \sum_{k=1}^d b_k^0 a_0 c_k^0 \psi_m\left( \sum_{w \in \mathcal{V}} w\right) = z \sum_{k=1}^d b_k^0 a_0 c_k^0 = \alpha \alpha^*.
\end{equation} 
Note also that  $\| y_0 \| \leq  ||c^0||_1$ and by \eqref{ineq:z<2}, $\| x_0 \|< 2 ||b^0||_1$. Hence using \eqref{ineq:M} and \eqref{eq:daalpha},  we get
$$ \| x_0 a y_0 - \alpha \alpha^*\|= \| x_0 a y_0 - x_0 a_0 y_0 \| < 2 M \|a-a_0\| < \frac{1}{4}. $$
Moreover, by construction, $x_0\in\alpha\alpha^*\cO^p(E)_{-n}$ and $y_0\in \cO^p(E)_n\alpha\alpha^*$, so $x_0ay_0\in \alpha\alpha^* \cO^p(E)_{0}\alpha\alpha^*$. Hence there is a $ b \in \alpha\alpha^*\cO^p(E)_{0}\alpha\alpha^*$ such that $b x_0 a y_0 = x_0 a y_0 b=\alpha \alpha^*$. Then for $h=n+|\alpha|$, $x = \alpha^*x_0 \in \cO^p(E)_{-h}$ and $y = y_0 b \alpha \in \cO^p(E)_h$ satisfy $x a y = r(\alpha)$.
\end{proof}

\begin{lema}\label{lem:trunca}
    Let $E$ be a purely infinite simple graph, $b\in L(E)$ and $v\in E^0$. Assume that $\Phi_0(b)=0$ and that $v$ is the basis of a cycle. Then there is a path $\sigma$ with $s(\sigma)=r(\sigma)=v$ and $\sigma^*b\sigma=0$.
\end{lema}

\begin{proof}
    As mentioned in the proof of \ref{lem:caminos} above, by cofinality of $E$ there are two closed paths based at $v$ which are incomparable with respect to the path order \eqref{orderpath}. Using them and proceeding as in the proof of 
    \cite{we}*{Lemma 9.5} one obtains an aperiodic infinite path $\theta$ starting at and passing through $v$ infinitely many times. We shall show that any finite closed path $\sigma$ with $\sigma\ge \theta$ and $|\sigma|$ sufficiently large satisfies the requirement of the lemma. By hypothesis, $b$ is a finite linear combination of elements of the form $\alpha\beta^*$ with $r(\alpha)=r(\beta)$ and $|\alpha|\ne |\beta|$. Since $\gamma^*\beta\alpha^*\gamma=(\gamma^*\alpha\beta^*\gamma)^*$ for any path $\gamma$, it suffices to show that if $|\alpha|>|\beta|$ then any $\sigma$ as above of sufficient length satisfies $\sigma^*\alpha\beta^*\sigma=0$. If $\alpha^*\theta=0$, this is clear. Otherwise $\alpha\ge \theta$ and is of the form $\alpha=\beta\alpha_1$ for some path $\alpha_1$ of positive length. Because $\theta$ is aperiodic, we may write $\theta=\beta\alpha_1^n\theta'$ with $\alpha_1^*\theta'=0$, so $\theta'=\alpha_2\theta''$ with $n\ge 1$, $\alpha_1^*\alpha_2=0$ and $r(\alpha_2)=v$. Hence if $\theta\le \sigma\le\beta\alpha^n\alpha_2$ is a closed path, then it must be of the form $\sigma=\beta\alpha^n\alpha_2\alpha_3$ and 
    $$\sigma^*\alpha\beta^*\sigma=\alpha_3^*\alpha_2^*(\alpha_1^n)^*\beta^*\beta\alpha_1\beta^*\beta\alpha_1^n\alpha_2\alpha_3=\alpha_3^*\alpha_2^*\alpha_1\alpha_2\alpha_3=0.$$ 
\end{proof}

\begin{proof}[Proof of Theorem \ref{teo:OE-sip}, part 2: row-finite SPI graphs without sources]\label{prop:pisws}
Let $E$ be a
\goodbreak
\noindent  purely infinite simple row-finite graph without sources; we want to prove that $\mathcal{O}^p(E)$ is purely infinite simple as a Banach algebra.
Since $L(E)$ is purely infinite simple and dense in $\mathcal{O}^p(E)$, it is enough to show 
\begin{equation}\label{toprove:pisws}
(\forall    0\ne a \in \mathcal{O}^p(E))\ \ (\exists x,y\in\cO^p(E),\ \ v\in E^0)\ \ x a y = v.
\end{equation}

Let $\Phi_0:\cO^p(E)\to \cO^p(E)_0$ be as in \eqref{map:phin}. 
Put $a_0=\Phi_0(a)$.  The proof will proceed in several steps, as follows. 

\medskip

    \item Step 1: $a_0=a$.  Then \eqref{toprove:pisws} holds by 
Lemma \ref{lem:caso0}.

\smallskip

\item Step 2: $a_0= v$ for some $v \in E^0$. Take $c\in L(E)$ such that $\| a - v - c\| < \frac{1}{4}$. Then $\Phi_0(c) \in L(E)$ and $\|\Phi_0 (c) \| =  \| \Phi_0 (a-v - c)\| < \frac{1}{4}$. Set $b = c - \Phi_0(c)$. By construction, $\|a - v - b\| < \frac{1}{2}$. Because $E$ is cofinal, there is a path $\eta$ with $s(\eta)=v$ such that $r(\eta)$ is the basis of a cycle. Moreover by Lemma \ref{lem:record}, $\Phi_0(\eta^*a\eta)=\eta^*\Phi_0(a)\eta=r(\eta)$. 
Thus, upon substituting $a'=\theta^*a\theta$ for $a$, we may assume that $v$ is the basis of a cycle.
By Lemma \ref{lem:trunca}, there is a closed path $\sigma$ based at $v$ such that $(\sigma)^* b \sigma = 0$. In particular, $\sigma\in v\cO^p(E)v$; moreover $\| (\sigma)^* a \sigma - v \| = \| (\sigma)^* a \sigma - (\sigma)^* v \sigma  \| = \| (\sigma)^*( a - v - b) \sigma  \| < \frac{1}{2}$. Thus there is a $z \in v\mathcal{O}^p(E)v$ such that $z \sigma^* a \sigma = v$.
This concludes the case when $a_0=v$.

\smallskip

\item Step 3: $a_0\ne 0$. By Lemma \ref{lem:caso0}, there exist $h\ge 0$, $x_0 \in \cO^p(E)_{-h}$, $y_0 \in \cO^p(E)_h$ such that $x_0 a_0 y_0 = v$. By  Lemma \ref{lem:record}, we also have $\Phi_0 ( x_0 a y_0)=v$. By the above case, there are $x, y \in \mathcal{O}^p(E)$ such that $x x_0 a y_0 y = v $. 

\smallskip

\item Step 4: $a \neq 0$. If $a_0\ne 0$ we are in the situation of Step 3 above. Otherwise, by  Lemma \ref{lem:record} (ii) there exists $n \in \Z\setminus\{0\}$ such that $\Phi_n(a) \neq 0$, and then because the finite sums of vertices of $E$ form an approximate unit for $\cO^p(E)$, there must be a $w\in E^0$ such that $\Phi_n(a) w \neq 0$. Assume $n>0$; because $E$ has no sources, there exists $\alpha \in \path$ such that $|\alpha| = n$ and $r(\alpha) = w $. Because $\Phi_0 (a \alpha^*) \alpha = \Phi_n(a) \alpha^* \alpha = \Phi_n(a) w \neq 0$, we get that $\Phi_0(a\alpha^*) \neq 0 $ and by the above case, there exist $x, y \in \mathcal{O}^p(E)$ such that $xa\alpha^*y = v$. If $n < 0 $, the argument is similar.
\end{proof}

\section{Proof of Theorem \ref{thm:intro}}\label{sec:final}

\begin{proof}
By Theorem \ref{thm:record}, the (purely infinite) simplicity of $L(E)$ is equivalent to that of $E$. By Theorems \ref{teo:E-sip} and \ref{teo:OE-sip}, $\cO^p(E)$ is SPI if and only if $E$ is. By definition a simple graph $E$ is not SPI if and only if it is acyclic. Hence if $\cO^p(E)$ is simple and not purely infinite, then $E$ is simple and acyclic, by Theorem \ref{teo:E-sip} and the second assertion of part (i) of Theorem \ref{thm:record} . The converse follows from Proposition \ref{prop:simpacyc} and Theorem \ref{teo:E-sip}. This concludes the proof. 
\end{proof}

\end{document}